\documentclass[9pt]{article}
\usepackage{amsthm}
\usepackage{amssymb}
\usepackage{amsmath}
\usepackage{multicol}
\usepackage{color}
\usepackage{graphicx}
\usepackage{amscd}
\usepackage[all]{xy}
\usepackage{comment}

\newtheorem{Thm}{Theorem}[section]

\newtheorem{Prop}[Thm]{Proposition}

\newtheorem{Cor}[Thm]{Corollary}
\newtheorem{Ques}[Thm]{Question}

\newcommand{\Ext}{\mathrm{Ext}}
\newcommand{\obs}{\mathrm{obs}}
\newcommand{\Ker}{\mathrm{Ker}}
\newcommand{\rank}{\mathrm{rank}}
\newcommand{\Hom}{\mathrm{Hom}}

\newcommand{\sht}{\mathrm{ht}^s}
\newcommand{\length}{\mathrm{length}}
\newcommand{\height}{\mathrm{ht}}

\newcommand{\Het}{H_{\text{\'et}}}

\newcommand{\C}{\mathbb{C}}

\newcommand{\G}{\mathbb{G}}

\renewcommand{\O}{\mathcal{O}}

\theoremstyle{definition}
\newtheorem{Def}[Thm]{Definition}
\newtheorem{Rem}[Thm]{Remark}

\author{Fuetaro Yobuko}

\title{Quasi-Frobenius splitting and lifting of Calabi-Yau varieties in characteristic $p$}
\date{}

\begin{document}
\maketitle

\begin{abstract}
Generalizing the notion of Frobenius-splitting, we prove that every finite height Calabi-Yau variety defined over an algebraically closed field of positive characteristic can be lifted to the ring of Witt vectors of length two.
\end{abstract}

\section{Introduction}

Let $k$ be an algebraically closed field of characteristic $p >0$.
A smooth proper variety $X$ over $k$ of dimension $n$ is said to be \textit{Calabi-Yau} if $\omega _X \simeq \mathcal{O}_X$ and $H^i(X, \mathcal{O}_X) = 0$ for $0 < i < n$.
It is known that there are some Calabi-Yau threefolds in positive characteristic which have no liftings to characteristic zero (\cite{H}, \cite{S}, \cite{HIS1}, \cite{HIS2}, \cite{Schoen}, \cite{CSt}).
This phenomenon is in contrast to the two dimensional case, i.e., every $K3$ surface over $k$ admits a lifting to characteristic zero (\cite{RS}, \cite{D}).

Furthermore, the Bogomolov-Tian-Todorov theorem says that the deformation functor of a Calabi-Yau variety over $\C$ has no obstructions (\cite{B}, \cite{To}, \cite{Ti}, \cite{R}, \cite{K}).
The non-existence of liftings to characteristic zero says that, in general, the deformation functor of a Calabi-Yau variety in mixed characteristic has a non-trivial obstruction.
In \cite{K}, Kawamata developes the \textit{$T^1$-lifting} method which is more algebraic than the methods in \cite{B}, \cite{To} and \cite{Ti} to show the smoothness of a deformation functor in characteristic zero.
Schr\"oer \cite{S1} studies the $T^1$-lifting property in mixed characteristic using divided power structures and Ekedahl and Sheperd-Barron \cite{ES} introduces a similar concept, \textit{(divided power) tangent lifting} property.
In particular, they give some conditions to ensure the smoothness of deformations of Calabi-Yau varieties.
But the sufficient conditions are hard to check in general and the situations in positive or mixed characteristic are not completely understood.

Recall that, for a Calabi-Yau variety $X$ over $k$, one has an invariant, the \textit{Artin-Mazur height} $\height (X)$, which takes a value in positive integers or infinity.
For the definition and its properties, see \S 3.

What is important for us is that, for all Calabi-Yau varieties which are known to be non-liftable, the Artin-Mazur heights are infinity.
Therefore, one can ask the following:

\begin{Ques}
If the Artin-Mazur height of a Calabi-Yau variety over $k$ is finite, does it admit a lifting to characteristic zero?
\end{Ques}

Let $W_2(k)$ be the ring of Witt vectors of length two of $k$.
Our main result gives a partial answer to the above question:

\begin{Thm}\label{main}
Let $X$ be a Calabi-Yau variety over $k$.
If the Artin-Mazur height of $X$ is finite, then $X$ admits a smooth lifting to $W_2(k)$. 
\end{Thm}

One can apply results of Deligne and Illusie \cite{DI} to $X$.
In particular, we have the following:

\begin{Cor}
If $\dim(X) \leq p$, the Hodge to de Rham spectral sequence of $X$ degenerates at $E_1$.
\end{Cor}

\begin{Rem}
In \cite{E2}, Ekedahl proves that Hirokado and Schr\"oer's non-liftable Calabi-Yau varieties admit no liftings to $W_2(k)$.
\end{Rem}

\begin{Rem}
Theorem \ref{main} supports a conjecture \cite[Conjecture 7.4.2]{J2} by Joshi on the existence of liftings of Calabi-Yau threefolds.
Let $X$ be a smooth projective Calabi-Yau threefold over $k$.
We denote by $c_3(X)$ the third Chern number of $X$ and by $b_i(X)$ the $i$-th Betti number of $X$.
Based on a detailed study of Hodge-Witt numbers $h_W^{i, j}$ (for their definition, see \cite[Chapter IV]{E1}), he conjectures that $X$ lifts to characteristic zero if and only if $H^0(X, \Omega_X)=0$ and $c_3(X) \leq 2 b_2(X)$. 

The later condition is equivalent to $h_W^{1, 2} \geq 0$ by \cite[Proposition 7.2.1]{J2}.
Note that this is the same as $b_3(X)$ being nonzero by \cite[Theorem 7.3.1]{J2}. 
Furthermore, by \cite[Theorem 6.1]{J1}, we know that the Artin-Mazur height of $X$ is finite if and only if $X$ is Hodge-Witt, that is, the slope spectral sequence degenerates at $E_1$-stage.
Again by \cite[Theorem 7.3.1]{J2}, we know that Hodge-Wittness implies the non-negativity of $h_W^{1, 2}$.
\end{Rem}

Theorem \ref{main} is known for Calabi-Yau varieties of height one.
Recall that a variety $X$ over $k$ is \textit{Frobenius-split} if the absolute Frobenius map
\[
F \colon \O_X \to F_{\ast}\O_X
\] 
splits as $\O_X$-modules. 
This is introduced in \cite{MR} and has applications to the representation theory of semisimple algebraic groups.
Furthermore, the same notion is actively studied in the theory of singularities (it is also called \textit{F-pure}).
For a Calabi-Yau variety $X$, the height of $X$ is one if and only if $X$ is Frobenius-split. 
In \cite{J1}, it is proved that any smooth Frobenius-split variety admits a flat lifting to $W_2(k)$.

A main ingredient of this paper is to introduce the notion of \textit{quasi-Frobenius-split} variety.
More precisely, we first define a new invariant $\sht(X)$, the \textit{Frobenius-split height} of a variety $X$, which quantifies the notion of Frobenius-splitting.
Then, we define that $X$ is quasi-Frobenius-split when $\sht(X)$ is finite.
In \S 4, we will show that any smooth quasi-Frobenius-split variety admits a lifting to $W_2(k)$ and, for a Calabi-Yau variety, the Artin-Mazur height is equal to the Frobenius-split height. 
For the first statement, the proof is done by the same line of \cite{J1}.
The second assertion is based on the study of the Artin-Mazur height of Calabi-Yau varieties by van der Geer and Katsura \cite{GK}.

\section*{Notations}
Throughout this paper, $k$ denotes an algebraically closed field of characteristic $p>0$.
We denote by $W(k)$ the ring of Witt vectors of $k$ and by $W_m(k)$ the ring of Witt vectors of length $m \geq 1$.
Note that $W_m(k) = W(k)/p^mW(k)$.
For a scheme $X$, $\Hom_{\O_X}$ and $\Ext_{\O_X}^i$ between $\O_X$-modules are shortened as $\Hom$ and $\Ext^i$.
For a scheme $X$ over $k$, $\Omega_X= \Omega_{X/k}^1$ denotes the sheaf of K\"ahler differentials on $X$.

\section{Differential calculus  and deformation theory in characteristic $p$}
Let $X$ be a smooth scheme over $k$ and $F$ the absolute Frobenius of $X$.
We first recall some properties of Illusie sheaves and Serre's Witt vector sheaves.

Let $B_1\Omega_X$ (resp. $Z_1\Omega _X$) be the sheaf of exact (resp. closed) one forms on $X$.
Since $d(f^pg)=f^pdg$, we may regard $B_1\Omega_X$ and $Z_1\Omega_X$ as $\O_X$-submodules of $F_{\ast}\Omega_X$.
We have the Cartier operator $C$ which fits into the following exact sequence;
\begin{align*}
0 \to B_1\Omega _X \to Z_1\Omega _X \xrightarrow{C} \Omega _X \to 0. 
\label{Cartier operator exact seq}
\end{align*}

Following Illusie \cite{L}, for $m \geq 1$, we define abelian subsheaves $B_m\Omega _X, Z_m\Omega _X$ of $\Omega_X$ inductively by  $B_{m+1}\Omega _X := C^{-1}(B_{m}\Omega _X)$, $Z_{m+1}\Omega _X := C^{-1}(Z_{m}\Omega _X)$.
We regard $B_{m}\Omega _X$ and $Z_{m}\Omega _X$ as $\mathcal{O}_X$-modules so that the inclusions $B_{m}\Omega _X \subset F_{\ast}^m\Omega _X$ and $Z_{m}\Omega _X \subset F_{\ast}^m\Omega _X$ are morphisms of $\mathcal{O}_X$-modules, that is, $f \in \O_X$ acts on $\omega \in B_m\Omega_X$ (or $ \omega \in Z_m\Omega_X$) as $f^{p^m}\omega$.

Let $W_m\O_X$ be the sheaf of Witt vectors of length $m$.
This is a sheaf of rings on $X$ and a local section of $W_m\O_X$ is expressed as an $m$-tuple of regular functions on $X$.
Note that $W_1\O_X$ is $\O_X$.
For each $m \geq 1$, we have three operators
\begin{align*}
&F \colon W_m\O_X \to W_m\O_X ; (f_0, \dots, f_{m-1})\mapsto (f_0^p, \dots, f_{m-1}^p), \\
&V \colon W_m\O_X \to W_{m+1}\O_X ; (f_0, \dots, f_{m-1}) \mapsto (0, f_0, \dots, f_{m-1}), \\
&R \colon W_{m+1}\O_X \to W_{m}\O_X ; (f_0, \dots, f_{m}) \mapsto (f_0, \dots, f_{m-1}).
\end{align*}

The Illusie sheaf $B_m\Omega_X$ is related to $W_m\O_X$ via the following exact sequence, due to Serre \cite{Ser};
\[
0 \to W_m\O_X \xrightarrow{F} F_{\ast}W_m\O_X \xrightarrow{D_m} B_m\Omega_X \to 0.
\]
Here $D_m$ is defined by the formula
\[
D_m \colon (f_0, \dots, f_{m-1})\mapsto df_{m-1} + f_{m-2}^{p-1}df_{m-2} + \cdots + f_0^{p^{m-1}-1}df_0.
\]

Next we recall the lifting theory by Nori-Srinivas as in Appendix of \cite{MS}.
By a \textit{lifting} of $X$ to $W_2(k)$, we mean a Cartesian diagram
\[
\xymatrix{
X \ar[d] \ar[r] & \tilde{X} \ar[d] \\
 \mathrm{Spec}(k) \ar[r] & \mathrm{Spec}(W_2(k)) }
\]
where the right vertical arrow is flat and the bottom horizontal arrow is induced by the natural surjection $W_2(k) \twoheadrightarrow k$.
By a \textit{lifting of the pair} $(X, F)$ to $W_2(k)$, we mean a pair $(\tilde{X}, \tilde{F})$ where $\tilde{X}$ is a lifting of $X$ and $\tilde{F} \colon \tilde{X} \to \tilde{X}$ is a morphism whose restriction to $X$ is $F$ and which is compatible with the Frobenius of $W_2(k)$.

The obstruction class $\obs_X$ for existence of a lifting of $X$ to $W_2(k)$ lives in $\Ext^2(\Omega_X, \O_X)$. 
Furthermore, in the appendix of \cite{MS}, it is proved that one has the obstruction class $\obs_{X, F}$ in $\Ext^1(\Omega_X, B_1\Omega_X)$ for existence of a lifting of the pair $(X, F)$ to $W_2(k)$. 
This means that $\obs_{X, F}$ is zero if and only if there exists a lifting of the pair $(X, F)$ to $W_2(k)$. 

In the following, we will use the same symbol for a derivation of $\O_X$ and the corresponding homomorphism from $\Omega_X$.

\begin{Prop}\label{local}
Let $\tilde{X}$ be a lifting of $X$ to $W_2(k)$.
\begin{description}
\item{(i)}
Let $\varphi $ be an infinitesimal automorphism of $\tilde{X}$, that is, an automorphism of $\tilde{X}$ over $W_2(k)$ which induces the identity on $X$.
Then there exists $\psi \in \mathrm{Hom}(\Omega_X, \O_X)$ such that $\varphi = id + p\psi$. 

\item{(ii)}
Let $\tilde{F}_1, \tilde{F}_2 $ be two liftings of the absolute Frobenius of $X$ to $\tilde{X}$.
Then there exists $\eta \in \mathrm{Hom}(\Omega_X, F_{\ast }\O_X)$ such that $\tilde{F}_1- \tilde{F}_2 =p\eta $.

\item{(iii)}
Let $\varphi$ and $\psi$ be as in (i). 
Let $\tilde{F}$ be a lifting of the absolute Frobenius.
Then $\varphi \tilde{F} \varphi ^{-1}$ is also a lifting of Frobenius and we have $\tilde{F} - \varphi \tilde{F} \varphi ^{-1} = p\psi ^p$.
\end{description}
\end{Prop}

\begin{proof}
The first statement is standard. 
For the others, see Proposition $1$ of Appendix of \cite{MS} and its proof.
\end{proof}

For any open subscheme $U \subset X$, we denote by $\mathrm{Rel}(X, F)(U)$ the set of isomorphism classes of liftings of the pair $(U, F|_{U})$ to $W_2(k)$.
Then they form a Zariski sheaf $\mathrm{Rel}(X, F)$ on $X$, which is a torsor under $\mathcal{H}om(\Omega_X, B_1\Omega_X)$ (see $ibid.$).

Similarly for any $U \subset X$, we denote by $\mathrm{sc}(C)(U)$ the set of sections of the Cartier operator $C \colon Z_1\Omega_U \to \Omega_U$, i.e., $\O_X$-linear morphisms $\phi \colon \Omega_U \to Z_1\Omega_U$ such that the composition $C \circ \phi$ is the identity of $\Omega_U$. 
Then the sheaf $\mathrm{sc}(C)$ is also a torsor under $\mathcal{H}om(\Omega_X, B_1\Omega_X)$.

\begin{Prop}\label{def theory}
\
\begin{description}
\item{(i)}
We have an isomorphism of torsors 
\begin{align*}
\mathrm{Rel}(X, F) \simeq \mathrm{sc}(C).
\end{align*}

\item{(ii)}
The class of the extension 
\begin{align*}
0 \to B_1\Omega_X \to Z_1\Omega_X \xrightarrow{C} \Omega_X \to 0
\end{align*}
is equal to $\obs_{X, F}$ in $\Ext^1(\Omega_X, B_1\Omega_X)$.

\item{(iii)}
Let $\mathcal{E}_1$ be the extension
\[
0 \to \O_X \xrightarrow{F} F_{\ast}\O_X \xrightarrow{d} B_1\Omega_X \to 0
\]
of $B_1\Omega_X$ by $\O_X$ and $\delta \colon \Ext^1(\Omega_X, B_1\Omega_X) \to \Ext^2(\Omega_X, \O_X)$ the connecting homomorphism induced by $\mathcal{E}_1$.
Then we have $\delta(\obs_{X, F}) = \obs_X$.

\end{description}
\end{Prop}

\begin{Rem}
The first assertion is a rigidified version of \cite[Theorem 3.5]{DI}.
The others are stated in \cite{Sri} without a proof.
\end{Rem}

\begin{proof}
(i) We have a morphism $\mathrm{Rel}(X, F) \to \mathrm{sc}(C)$ defined by the following.
Let $(\tilde{U}, \tilde{F})$ be a lifting of $(U, F|_{U})$.
For a local section $x \in \O_{\tilde{U}}$, we have $\tilde{F}(x) = x^p + p\psi(\overline{x})$ for some function $\psi \colon \O_{U} \to \O_{U}$.
Here $\overline{x} \in \O_{U}$ is the image of $x$ under the reduction modulo $p$.
An assignment $\overline{x} \mapsto \overline{x}^{p-1}d\overline{x} + d\psi(\overline{x})$ defines a section $\phi$ of the Cartier operator over $U$.
This is a morphism of torsors under $\mathcal{H}om(\Omega_X, B_1\Omega_X)$, so these two are isomorphic.

(ii) This follows from (i).

(iii) We first recall the construction of the obstruction classes $\obs_X$ and $\obs_{X,F}$. 
Take an affine open covering $X = \bigcup U_i$ such that there exist liftings $(\tilde{U}_i, \tilde{F}_i)$ of the pair $(U_i, F|_{U_i})$ for each $i$.
Let $\tilde{U}_{ij} \subset \tilde{U}_i$ be the open subscheme corresponding to $U_{ij} = U_i \cap U_j\subset U_i$.
Then we have isomorphisms $\varphi_{ij} \colon \tilde{U}_{ij} \xrightarrow{\sim} \tilde{U}_{ji}$ such that $\varphi_{ij}|_{U_{ij}}$ is the identity.
In the following, we omit the symbol of the restriction to a smaller open subscheme.
The composition $\varphi_{ik}^{-1} \varphi_{jk} \varphi_{ij}$ is an infinitesimal automorphism of $\tilde{U}_{ijk}$ and defines a derivation $\varphi_{ijk} \in \mathrm{Hom}(\Omega_{U_{ijk}}, \O_{U_{ijk}})$.
Then the class $\{\varphi_{ijk}\}$ defines $\obs_X \in H^2(X, \mathcal{H}om(\Omega_X, \O_X)) = \Ext^2(\Omega_X, \O_X) $.

Now $\varphi_{ji} \tilde{F}_j \varphi _{ji}^{-1}$ and $\tilde{F}_i$ are two liftings of the Frobenius on $U_{ij}$.
By Proposition \ref{local}(ii), they differ by some $\eta _{ij} \in \mathrm{Hom}(\Omega_{U_{ij}}, F_{\ast}\O_{U_{ij}})$ and its image $\overline{\eta}_{ij} \in \mathrm{Hom}(\Omega_{U_{ij}}, B_1\Omega_{U_{ij}})$ of $\eta_{ij}$ is independent of the choice $\varphi_{ij}$. 
Then $\{{\overline{\eta}_{ij}}\}$ defines the obstruction class $\obs_{X, F} \in \Ext^1(\Omega_X, B_1\Omega_X) = H^1(X, \mathcal{H}om(\Omega_X, B_1\Omega_X))$.

We have $\eta_{ij} + \eta_{jk}-\eta_{ik} = \eta _{ijk}^p$ for some $\eta_{ijk} \in \mathrm{Hom}(\Omega_{U_{ijk}}, \O_{U_{ijk}})$ and the class $\{ \eta_{ijk} \}$ represents $\delta(\obs_{X, F})$.
By definition, we have $\varphi_{ji} \tilde{F}_j \varphi _{ji}^{-1} - \tilde{F}_i = p\eta_{ij}$.
From these equations, we obtain
\begin{align*}
\varphi_{ik}^{-1} \varphi_{jk} \varphi_{ij} \tilde{F}_i \varphi _{ij}^{-1} \varphi _{jk}^{-1} \varphi _{ik} - \tilde{F}_i = p(\eta_{ij} + \eta_{jk} - \eta_{ik}).
\end{align*}
for any $i, j, k$.
By Proposition \ref{local}, we see that $\varphi_{ijk}^p = \eta_{ijk}$. 
Therefore, $\delta(\obs_{X, F}) = \obs_X$.
\end{proof}

\begin{Cor}
Let 
\[
\Ext^1(\Omega_X, B_1\Omega_X) \otimes \Ext^1(B_1\Omega_X, \O_X) \xrightarrow{(-, -)} \Ext^2(\Omega_X, \O_X)
\]
be the Yoneda paring.
Then we have
\[
(\obs_{(X, F)}, \mathcal{E}_1) = \obs_X.
\]
\end{Cor}

\begin{proof}
This is a restatement of Proposition \ref{def theory}(iii).
\end{proof}

\begin{Rem}
The lifting theory in this section is valid for the lifting problem of $(\tilde{X}, \tilde{F})$ over $W_n(k)$ to $W_{n+1}(k)$ for any $n$. 
\end{Rem}

\section{Artin-Mazur height of Calabi-Yau varieties after van der Geer and Katsura}

Let $X$ be a scheme over $k$ and $n$ the dimension of $X$.
Consider the functor $\Phi_X$ from the category of Artin local $k$-algebras with residue field $k$ to the category of abelian groups defined by
\[
\Phi_X \colon A \mapsto \Ker (\Het^n(X \otimes _kA, \G_m) \to \Het^n(X, \G_m)).
\]

By the results of Artin and Mazur \cite{AM}, when $X$ is a Calabi-Yau variety, the functor $\Phi_X$ is pro-represented by a one dimensional formal group.
The \textit{Artin-Mazur height} $\height(X)$ of $X$ is defined to be the height of the associated formal group $\Phi_X$.
The Dieudonn\'e module of $\Phi_X$ is canonically isomorphic to the Serre's Witt vector cohomology $H^n(X, W\O_X) := \varprojlim_m H^n(X, W_m\O_X)$.
In particular, we see that
\begin{align*}
\height(X) =
\begin{cases}
\dim _{K} H^n(X, W\O_X) \otimes K & \text{if $H^n(X, W\O_X) \otimes K \neq 0$}, \\
\infty & \text{if $H^n(X, W\O_X) \otimes K = 0$}.
\end{cases}
\end{align*}
Here $K$ is the field of fractions of $W(k)$.
In the case $\height (X)< \infty$, the $W(k)$-module $H^n(X, W\O_X)$ is free and finitely generated.

This invariant has the following characterization due to van der Geer and Katsura \cite{GK}. 

\begin{Thm}
Let $X$ be a Calabi-Yau variety over $k$ of dimension $n$.
Then $\height(X)$ is equal to the minimum number $m$ such that $F$ acts non-trivially on $H^n(X, W_m\O_X)$.
\end{Thm}

In the following, the symbol “$\dim$” denotes the dimension of a $k$-vector space.
In general, $H^i(X, W_m\O_X)$ is not a $k$-vector space, but a $W(k)$-module of finite length (when $X$ is proper over $k$).
The symbol “$\length$” denotes the length as a $W(k)$-module.
Since the operators $F$ and $V$ on $H^i(X, W_m\O_X)$ satisfies relations $FV=p=VF$, the $W(k)$-module $H^i(X, W_m\O_X)/FH^i(X, W_m\O_X)$ is a $k$-vector space and one can speak of its dimension as a $k$-vector space.

There is another formula which will be useful for our purpose.

\begin{Cor}[\cite{GK} Proposition 3.1.]\label{dim formula}
Let $X$ be a Calabi-Yau variety over $k$ of dimension $n$.
Then we have
\begin{align*}
\dim H^i(X, B_m\Omega _X) = 
\begin{cases}
0 & \text{if $i \neq n-1, n$ }, \\
\min\{m, \height(X) -1\} & \text{if $i = n-1, n$}.
\end{cases}
\end{align*}
\end{Cor}

\begin{proof}
For the sake of completeness, we give a proof.
By the vanishing condition of $H^i(X, \O_X)$ for $0 < i<n$, it is easy to see that $H^i(X, B_m\Omega_X)=0$ for $i < n-1$.

Furthermore, by the exact sequence $\mathcal{E}_1$ (for the definition of $\mathcal{E}_1$, see Proposition \ref{def theory}(iii) or later Remark \ref{E_m}), we see that
\[
\chi (B_1\Omega_X) := \sum_i (-1)^i\dim H^i(X, B_1\Omega_X) = \chi(F_{\ast}\O_X) -\chi (\O_X)=0.
\]
Note that, in the last equality, we used the fact that $F$ is a finite morphism.
Since $B_m\Omega_X$ is a successive extension of $B_1\Omega_X$, we see that $\chi(B_m\Omega_X)=0$ for any $m \geq 1$.
Therefore, it is  enough to show that $\dim H^n(X, B_m\Omega _X) = \min\{m, \height(X) -1\}$.

Since the Cartier operator induces a surjection $B_m\Omega_X \to B_{m-1}\Omega_X$, the sequence $\{\dim H^n(X, B_m\Omega _X)\}_{m \geq 1}$ is increasing. 
If $m < \height(X) $, by the theorem, we see that
\begin{align*}
\dim H^n(X, B_m\Omega _X) & = \dim \frac{H^n(X, W_m\O_X)}{FH^n(X, W_m\O_X)} \\
&= \length H^n(X, W_m\O_X) =m.
\end{align*}
Furthermore, when $\height (X)$ is finite,  we have
\begin{align*}
\dim H^n(X, B_m\Omega _X) & = \dim \frac{H^n(X, W_m\O_X)}{FH^n(X, W_m\O_X)} \\
& \leq \dim \frac{H^n(X, W\O_X)}{FH^n(X, W\O_X)} \\
& = \rank_W H^n(X, W\O_X) - \dim \frac{H^n(X, W\O_X)}{VH^n(X, W\O_X)} \\
& = \height (X) -1.
\end{align*}
In the second equality, we used the fact that $H^n(X, W\O_X)$ is a free $W(k)$-module and equalities $p=FV=VF$.
This completes the proof.
\end{proof}

\section{Frobenius-split height and quasi-Frobenius-split varieties}

Let $X$ be a scheme over $k$.
For each $m >0$, consider the following morphisms between $W_m\O_X$-modules:
\[
\xymatrix{
W_m\O_X \ar[d]_-{R^{m-1}} \ar[r]^-{F} & F_{\ast}W_m\O_X \\
\O_X.
}
\]

\begin{Def}\label{q f split}
We define the \textit{Frobenius-split height} $\sht(X)$ of $X$ by the minimum number $m>0$ such that there exists a $W_m\O_X$-linear homomorphism 
\begin{align*}
\phi \colon F_{\ast}W_m\O_X \to \O_X
\end{align*}
satisfying $\phi \circ F = R^{m-1}$.
If such an $m$ does not exist, then $\sht(X)$ is $\infty$. 
A scheme $X$ over $k$ is said to be \textit{quasi-Frobenius-split} if $\sht(X) < \infty$.
\end{Def}

\begin{Rem}
By the definition, $X$ is Frobenius-split if and only if $\sht(X) = 1$.
\end{Rem}

\begin{Rem}\label{E_m}
Assume that $X$ is smooth over $k$.
Consider the following exact sequence;
\begin{align*}
0 \to W_m\O_X \xrightarrow{F} F_{\ast}W_m\O_X \xrightarrow{D_m} B_m\Omega_X \to 0 . 
\end{align*}
If we pushout this sequence via $W_{m}\O_X \xrightarrow{R^{m-1}} \O_X$, we get a new exact sequence
\begin{align*}
0 \to \O_X \to \mathcal{E}_m \to B_m\Omega_X \to 0. 
\end{align*}
We denote by $\mathcal{E}_m$ this extension of $B_m\Omega_X$ by $\O_X$.
When $m=1$, the morphism $R^{m-1}$ is the identity of $\O_X$ and $\mathcal{E}_1$ is equal to the one previously defined in Proposition \ref{def theory}(iii).
The existence of $\phi$ as in Definition \ref{q f split} is equivalent to the splitting of the exact sequence $\mathcal{E}_m$ and the Frobenius-split height is characterized as
\begin{align*}
\sht(X) = \min\{m>0 ; \text{$\mathcal{E}_m$ splits}\}.
\end{align*}
Furthermore, since the diagram  
\[
\xymatrix{
F_{\ast}W_{m+1}\O_X  \ar[d]_{R}\ar[r]^-{D_{m+1}} & B_{m+1}\Omega_X \ar[d]^{C} \\
F_{\ast}W_m\O_X \ar[r]^-{D_m} & B_m\Omega_X
}
\]
commutes, the exact sequence $\mathcal{E}_m$ is obtained from the exact sequence
\begin{align*}
0 \to \O_X \to F_{\ast}\O_X \xrightarrow{d} B_1\Omega_X \to 0  
\end{align*}
by pulling back along $B_m\Omega_X \xrightarrow{C^{m-1}} B_1\Omega_X$, i.e., we have a morphism of extensions
\[
\xymatrix{
0 \ar[r] &\O_X \ar[r] \ar@{=}[d] &\mathcal{E}_m \ar[r] \ar[d] &B_m\Omega_X \ar[r] \ar[d]^{C^{m-1}} &0 &\mathcal{E}_m \\
0 \ar[r] &\O_X \ar[r] &F_{\ast}\O_X \ar[r] &B_1\Omega_X \ar[r] &0 &\mathcal{E}_1.
}
\]
\end{Rem}

The proof of our main theorem is divided into two parts. 

\begin{Thm}
Every smooth quasi-Frobenius-split scheme over $k$ admits a smooth lifting to $W_2(k)$. 
\end{Thm}

\begin{proof}
Let $\sht(X) = m < \infty$.
By the above remark, the extension class $\mathcal{E}_m$ is zero in $\Ext^1(B_m\Omega_X, \O_X)$.
We have the following diagram:
\[
\begin{xy}
(2, 0) *{\Ext ^1(\Omega_X, B_1\Omega_X)}, 
(30, 0)*{\otimes \Ext^1(B_1\Omega_X, \O_X)},
(65, 0) *{\Ext^2(\Omega_X, \O_X).}, 
(1, 10) *{\Ext ^1(\Omega_X, B_m\Omega_X)}, 
(30, 10)*{\otimes \Ext^1(B_m\Omega_X, \O_X)},
(65, 10) *{\Ext^2(\Omega_X, \O_X)}, 
\ar @{=} (65, 3); (65, 7)
\ar @{>} (46, 0); (53, 0)_{(-, -)}
\ar @{>} (46, 10); (53, 10)^{(-, -)}
\ar @{>} (10, 7) ; (10, 3)_{C_{\ast}^{m-1}}
\ar @{>} (20, 3) ;(20, 7)_{(C^{m-1})^{\ast}}
\end{xy}
\]
For any $x \in \Ext ^1(\Omega_X, B_m\Omega_X)$ and $y \in \Ext^1(B_1\Omega_X, \O_X)$, we have
\[
(C_{\ast}^{m-1}x, y) = (x, (C^{m-1})^{\ast}y).
\]
In particular, we have
\[
(C_{\ast}^{m-1}x, \mathcal{E}_1) = (x, \mathcal{E}_m).
\]

Consider the following commutative diagram 
\[
\xymatrix{
0 \ar[r] & B_m\Omega_X \ar[r] \ar[d]_{C^{m-1}}& Z_m\Omega_X \ar[r]^{C^m} \ar[d]_{C^{m-1}}& \Omega_X \ar[r] \ar@{=}[d]& 0 \\
0 \ar[r] & B_1\Omega_X \ar[r] & Z_1\Omega_X \ar[r]^{C} & \Omega_X \ar[r] & 0.
}
\]
Let $x \in \Ext ^1(\Omega_X, B_m\Omega_X)$ be the extension class defined by the upper exact sequence.
By Proposition \ref{def theory}, the lower exact sequence defines the obstruction class $\obs_{(X, F)} \in \Ext^1(\Omega_X, B_1\Omega_X)$.
This means $\obs_{(X, F)} = C_{\ast}^{m-1}x$.
Again by the same proposition, we see that
\begin{align*}
\obs_X &= (\obs_{(X, F)}, \mathcal{E}_1) \\
&= (C_{\ast}^{m-1}x, \mathcal{E}_1) \\
&= (x, \mathcal{E}_m) = 0.
\end{align*}
This completes the proof.
\end{proof}

\begin{Thm}
For a Calabi-Yau variety $X$ over $k$, we have 
\begin{align*}
\height(X) = \sht(X).
\end{align*}
\end{Thm}

\begin{proof}
It is enough to show that, for any $m \geq 1$, $\height (X) \leq m$ if and only if $\sht (X) \leq m$, i.e., the extension class $\mathcal{E}_m$ is zero.
Consider the exact sequence
\begin{align*}
0 \to F_{\ast }B_{m-1}\Omega _X \xrightarrow{\iota} B_{m}\Omega _X \xrightarrow{C^{m-1}} B_1\Omega _X \to 0.
\end{align*}
Then we have an exact sequence
\begin{align*}
\Ext^1(B_1\Omega _X, \O_X) \xrightarrow{(C^{m-1})^{\ast}} \Ext^1(B_{m}\Omega _X, \O_X) \xrightarrow{\iota^{\ast}} \Ext^1(F_{\ast }B_{m-1}\Omega _X, \O_X) \\ 
\to \Ext^2(B_1\Omega _X, \O_X).
\end{align*}
Recall that $\mathcal{E}_1 \in \Ext^1(B_1\Omega_X, \O_X)$ and $(C^{m-1})^{\ast}(\mathcal{E}_1) = \mathcal{E}_m$.
By the Serre Duality and triviality of the canonical bundle, we see $\Ext^i(B_m\Omega _X, \O_X) \simeq H^{n-i}(X, B_m\Omega _X)^{\vee}$.
Here $n$ denotes the dimension of $X$.
Similarly, since $F$ is finite, we compute as
\[
\Ext^i(F_{\ast }B_{m-1}\Omega _X, \O_X) \simeq H^{n-i}(X, F_{\ast}B_{m-1}\Omega _X)^{\vee} \simeq H^{n-i}(X, B_{m-1}\Omega _X)^{\vee}.
\] 
By Corollary \ref{dim formula}, the morphism $\iota^{\ast}$ is surjective for any $m$.

Since Frobenius-splitting is the same as the Artin-Mazur height being one, using Corollary \ref{dim formula}, we see that $\Ext^1(B_1\Omega _X, \mathcal{O}_X)$ is generated by the class of $\mathcal{E}_1$.
This implies that $\mathcal{E}_m$ is zero if and only if $\iota^{\ast}$ is an isomorphism.
By the same Corollary \ref{dim formula}, this is equivalent to $\height(X) \leq m$.
\end{proof}

\begin{Rem}
One of the most important properties of a Frobenius-split variety $X$ is that, for an ample line bundle $\mathcal{L}$ on $X$, we have $H^i(X, \mathcal{L}) = 0$ for $i>0$ (\cite{MR}).
The same property holds for quasi-Frobenius-split varieties.
\end{Rem}

\section*{Acknowledgments}
The author would like to express his sincere gratitude to his advisor Professor Nobuo Tsuzuki.
He thanks Professor Joshi Kirti for informing him the conjecture on the lifting of Calabi-Yau threefolds and explaining him the relation between the conjecture and this work.
He also thanks Professor Yukiyoshi Nakkajima and the referee for their careful readings of the manuscript and useful suggestions.
He was supported by Grant-in-Aid for JSPS Fellow 15J05073.

\bigskip
\bigskip
%\parno
\par\noindent
Fuetaro Yobuko
%\parno
\par\noindent
Mathematical Institute, 
Tohoku University,
6-3 Aoba  Aramaki Aoba-Ku Sendai,  
Miyagi 980--8578, Japan. 
\par \noindent
(Current address: Graduate School of Mathematics, Nagoya University, 
Furocho, Chikusaku, Nagoya, 464-8602, Japan.)
\par\noindent
\textit{E-mail address}:
soratobumusasabidesu@gmail.com

\end{document}